\documentclass[12pt,a4paper,fleqn]{article}
\usepackage{latexsym,amssymb,amsfonts,amsmath,amsthm,nccmath,enumitem}
\usepackage{array,color,colortbl}
\usepackage[margin=1.75cm]{geometry}

\newtheorem{theorem}{Theorem}
\newtheorem{corollary}[theorem]{Corollary}
\newtheorem{lemma}[theorem]{Lemma}

\def\etal{{\it et al.}~\cite}

\def\eref#1{$(\ref{#1})$}

\def\lref#1{Lemma~$\ref{#1}$}
\def\tref#1{Theorem~$\ref{#1}$}
\def\cyref#1{Corollary~$\ref{#1}$}

\renewcommand{\ge}{\geqslant}
\renewcommand{\le}{\leqslant}

\newcommand{\Z}{\mathbb{Z}}
\def\sym{{\cal S}}

\def\kee{\psi}
\def\ab{\psi_{+}}
\def\cyc{\psi_\circ}

\def\ev{E}

\def\id{\varepsilon}

\def \mod#1{{\:({\rm mod}\ #1)}}

\let\oldproofname=\proofname
\renewcommand{\proofname}{\rm\bf{\oldproofname}}

\title{Small Partial Latin Squares that\\
Cannot be Embedded in a Cayley Table}

\author{Ian M.\ Wanless\thanks{Research supported by ARC grant FT110100065.}\\
\small School of Mathematical Sciences\\[-0.8ex]
\small Monash University\\[-0.8ex]
\small Vic 3800 Australia\\
\small \texttt{ian.wanless@monash.edu}
\and
Bridget S. Webb\\
\small The Open University,\\[-0.8ex]
\small Milton Keynes, MK7 6AA,\\[-0.8ex]
\small United Kingdom\\
\small \texttt{bridget.webb@open.ac.uk}
}

\date{}

\begin{document}

\maketitle

\begin{abstract}
We answer a question posed by D\'enes and Keedwell that is equivalent to
the following.
For each order $n$ what is the smallest size of a 
partial latin square that cannot be embedded into the Cayley table
of any group of order $n$? We also solve some variants of this question
and in each case classify the smallest examples that cannot be embedded.
We close with a question about embedding of diagonal partial latin
squares in Cayley tables.
\end{abstract}


\section{Introduction}


A \emph{partial latin square} (PLS) is a matrix in which some cells may be 
empty and in which any two filled
cells in the same row or column must contain distinct symbols.  
In this work we will insist that each row and column of a PLS must
contain at least one filled cell.
The {\em size} of the PLS is the number of filled cells.
The {\em order} of the PLS is the maximum of its number of rows,
number of columns and the number of distinct symbols it contains.
We say that a PLS $P$ {\em embeds} in
a group $G$ if there is a copy of $P$ in the Cayley table of $G$.
Formally, this means there are
injective maps $I_1$, $I_2$ and $I_3$ from, respectively,
the rows, columns and symbols of $P$ to $G$ such that
$I_1(r)I_2(c)=I_3(s)$ whenever a symbol $s$ occurs in cell $(r,c)$
of $P$.

Let $\kee(n)$ denote the largest number $m$ such that for every
PLS $P$ of size $m$ there is some group of order $n$ in which $P$
can be embedded. Rephrased in our terminology, Open Problem 3.8 in 
\cite{DKI} asks for the value of $\kee(n)$. In this note we
solve this problem by showing:

\begin{theorem}\label{t:main}
\[
\kee(n)=\begin{cases}
1&\text{when }n=1,2,\\
2&\text{when }n=3,\\
3&\text{when }n=4,\text{ or when $n$ is odd and }n>3,\\
5&\text{when }n=6,\text{ or when }n\equiv2,4\mod 6\text{ and }n>4,\\
6&\text{when }n\equiv0\mod 6\text{ and }n>6.
\end{cases}
\]
\end{theorem}

We also consider an abelian variant.
Let $\ab(n)$ denote the largest number $m$ such that for every
PLS $P$ of size $m$ there is some {\em abelian} group of order $n$ in which $P$
can be embedded. We show that:

\begin{theorem}\label{t:abmain}
\[
\ab(n)=\begin{cases}
1&\text{when }n=1,2,\\
2&\text{when }n=3,\\
3&\text{when }n=4,\text{ or when $n$ is odd and }n>3,\\
5&\text{when $n$ is even and }n>4.
\end{cases}
\]
\end{theorem}

Narrowing the focus even further, let $\cyc(n)$ be
the largest number $m$ such that every
PLS $P$ of size $m$ embeds in the cyclic group $\Z_n$. We show that

\begin{theorem}\label{t:cycmain}
$\cyc(n)=\ab(n)$ for all positive integers $n$.
\end{theorem}

Note that it is immediate from the definitions that 
$\cyc(n)\le\ab(n)\le\kee(n)$ for all $n$.

In our investigations we will repeatedly use the observation that
pre- and post-multiplication allow us, without loss of generality,
to specify one row to be mapped to the identity by $I_1$ and
one column to be mapped to the identity by $I_2$ (see Lemma 2 in
\cite{CW09}). We will also often find it convenient to consider
the symbols in a PLS to be group elements (put another way, 
we may treat $I_3$ as the identity map).


We will use $\id$ to denote the identity element of a group, except
that when we know the group is abelian we will use additive notation
with $0$ as the identity.

\section{Upper bounds}

We first show that $\kee(n)$ never exceeds the values
quoted in \tref{t:main}. To do this we construct, for each $n$,
a PLS of size $\kee(n)+1$ that cannot be embedded into any
group of order $n$.

A \emph{quasigroup} $Q$ is a nonempty set with one binary operation
such that for every $a$, $b\in Q$ there is a unique $x\in Q$ and a
unique $y\in Q$ satisfying $ax=b=ya$.  The definition of an
embedding for a PLS converts without change from groups to
quasigroups.  The quasigroup analogue of finding $\kee(n)$ was the
subject of a famous conjecture known as the Evans conjecture. It is
now a theorem \cite{AH83,Sme81}.

\begin{theorem}\label{t:evansconj}
  Each PLS of size at most $n-1$ can be embedded into some quasigroup
  of order $n$. For each $n>1$ there is a PLS of size $n$ that cannot
  be embedded into any quasigroup of order $n$.
\end{theorem}

The difficult part of \tref{t:evansconj} is the first statement.
Examples that show the second statement are fairly obvious. For 
$1\le a<n$ a PLS $\ev_{n,a}$ of size $n$ can be constructed as follows:
\begin{gather}\label{e:evanstight}
\begin{aligned}
\ev_{n,a}(1,i)&=i\text{ for }1\le i\le a,\\
\ev_{n,a}(i,n)&=i\text{ for }a<i\le n.
\end{aligned}
\end{gather}
Clearly, none of the symbols $1,2,\dots,n$ is available to fill the
cell $(1,n)$, so $\ev_{n,a}$ cannot be embedded in a quasigroup of order
$n$. By \cite{AH83,Dam83}, this is essentially the only way to build a PLS
of size $n$ that cannot be embedded in a quasigroup of order $n$.
An immediate consequence of these examples is:

\begin{corollary}\label{cy:psilessn}
$\kee(n)<n$ for $n>1$.
\end{corollary}

There is another way in which \cyref{cy:psilessn} can be derived for
certain values of $n$, which again connects to important prior work. A
{\em complete mapping} for a group $G$ is a permutation $\phi$ of the
elements of $G$ such that the map $x\mapsto x\phi(x)$ is also a
permutation of the elements of $G$. See \cite{DKI,transurvey} for
context, including a proof that no group of order $n\equiv2\mod4$ has
a complete mapping.  This implies for $n\equiv2\mod4$ that
$\kee(n)<n$, since no group can have an embedding of the PLS $T_{n}$
of size $n$, where
\begin{equation}\label{e:compmap}
T_{n}(i,i)=i\text{ for }i=1,\dots,n.
\end{equation}
Later we need the following special case of a theorem by 
Brouwer \etal{BVW78} and Woolbright~\cite{Woolbright}.

\begin{theorem}\label{t:trans}
$T_t$ embeds in any group of order $n$, provided $t\le \lceil n-\sqrt n\rceil$.
\end{theorem}

Next we show:

\begin{lemma}\label{l:lcyc}
For each $\ell\ge2$ there exists a PLS of size $2\ell$ that
can only be embedded in groups whose order is divisible by $\ell$.
\end{lemma}

\begin{proof}
Consider the PLS
\[
C_\ell=
\left(
\begin{array}{ccccc}
a_1&a_2&\cdots&a_{\ell-1}&a_\ell\\
a_2&a_3&\cdots&a_{\ell}&a_1
\end{array}
\right)
\]
which is often known as a {\em row cycle} in the literature 
(see e.g.~\cite{Wan99,cycsw}). Suppose that $C_\ell$ is embedded in rows 
indexed $r_1$ and $r_2$ of the Cayley table of a group $G$.
From the regular representation of $G$ as used in Cayley's theorem, it
follows that $r_1^{-1}r_2$ has order $\ell$ in $G$. In particular
$\ell$ divides the order of $G$.
\end{proof}

Finally, we exhibit a constant upper bound.

\begin{lemma}\label{l:atmost6}
$\kee(n)\le6$ for all $n$.
\end{lemma}

\begin{proof}
The following pair of PLS of order 7
\begin{equation}\label{e:quadcrit}
\left(
\begin{array}{ccc}
a&b&\cdot \\
c&a&b \\
\cdot&c&d
\end{array}
\right)
\qquad
\left(
\begin{array}{ccc}
a&b&\cdot \\
c&a&b \\
\cdot&d&a
\end{array}
\right)
\end{equation}
each fail the so-called quadrangle criterion \cite{DKI} and hence
neither can be embedded into any group.
\end{proof}

\lref{l:lcyc} for $\ell\in\{2,3\}$ combined with \lref{l:atmost6} and
\cyref{cy:psilessn} gives the upper bounds on $\kee(n)$ 
that we set out to prove.


\section{Proof of the main results}

Let $\sym_n$ denote the symmetric group of degree $n$.  It can be
convenient to view a PLS as a set of {\em triples} of the form
$(r,c,s)$ which record that symbol $s$ occupies cell $(r,c)$. When
viewed in this way, there is a natural action, called {\em parastrophy},
of $\sym_3$ on PLS where the triples are permuted uniformly.  There is
also a natural action of $\sym_n\wr\sym_3$ on the set of PLS of order $n$.
Orbits under this action are known as {\em species} (the term
{\em main classes} is also used). See \cite{tradenum} for full
details.

To prove \tref{t:main} it remains to show that every PLS of size no
more than the claimed value of $\kee(n)$ can indeed be embedded in
some group of order $n$. We also aim to identify every species of PLS
of size $\kee(n)+1$ which cannot be embedded into any group of order
$n$.  We refer to such a PLS as an {\em obstacle for $\kee(n)$}. 
Obstacles for $\ab(n)$ and $\cyc(n)$ are defined similarly.  
By our work in the previous section we know that all
obstacles have size at most 7. A catalogue of species representatives
for all PLS of size at most 7 is simple to generate (it is much simpler
task than the enumeration in \cite{tradenum}, though some programs
from that enumeration were reused in the present work). 
The number of species involved is
shown in the following table:
\begin{equation}\label{e:numPLS}
\begin{array}{|c|ccccccc|}
\hline
\text{size}&1&2&3&4&5&6&7\\
\text{\#species}&1&2&5&18&59&306&1861\\
\hline
\end{array}
\end{equation}

It suffices to consider one representative of each species because of the
following Lemma. We omit the proof since it is identical to that
of Lemma 1 in \cite{CW09}, and an easy consequence of 
Theorem 4.2.2 in \cite{DKI}.

\begin{lemma}\label{l:specinvar}
Let $G$ be an arbitrary group and $P,P'$ two PLS from the 
same species. Then $P$ embeds in $G$ if and only if $P'$ embeds in $G$.
\end{lemma}

To find $\kee(n)$, $\ab(n)$ and $\cyc(n)$ we wish to identify the
smallest PLS that cannot be embedded into any group (resp.\ abelian
group, cyclic group) of order $n$.  The following two lemmas allow us
to eliminate many candidates from consideration.

We use $\bullet$ to denote an arbitrary row, column or symbol
(possibly a different one each time that $\bullet$ appears).

\begin{lemma}\label{l:rulenoopt}
Let $G$ be a group of order $n$ and $P$ a PLS of order at most $n$. 
If $P$ contains a triple $(r,c,s)$
with the following properties:
\begin{enumerate}
\item[\rm (i)]
$P':=P\setminus\{(r,c,s)\}$ has no triple $(r,\bullet,\bullet)$,
\item[\rm (ii)]
for each triple $(r',c',s')$ in $P'$ there is a triple
$(r',c,\bullet)$ or $(r',\bullet,s)$ in $P'$,
\end{enumerate}
then $P$ can be embedded in $G$ if $P'$
can be embedded in $G$.
\end{lemma}

\begin{proof}
Consider an embedding $(I_1,I_2,I_3)$ of $P'$ in $G$.
By condition (i) we know that $I_1(r)$ will not be defined.
If $I_2(c)$ is undefined then $P'$ must have fewer that $n$ columns
so we may simply define $I_2(c)$ to be any
element of $G\setminus\{I_2(y):(x,y,z)\in P'\}$. 
If $I_3(s)$ is undefined then we choose its value similarly.
Then we define $I_1(r)=I_3(s)I_2(c)^{-1}$. 
By construction, both $I_2$ and $I_3$ are injective.
So at this point, the only thing which
could prevent us having an embedding of $P$ in $G$ would be if 
$I_1(r)=I_1(r')$ for some $(r',c',s')\in P'$. Suppose this is the case.
Then by condition (ii) there is a triple 
$(r',c,\bullet)$ or $(r',\bullet,s)$ in $P'$. First suppose that
$(r',c,s'')$ is in $P'$ for some symbol $s''$.
Then $I_3(s'')=I_1(r')I_2(c)=I_1(r)I_2(c)=I_3(s)$ so
$s''=s$. But this means $(r',c,s)$ and $(r,c,s)$ are distinct triples
in $P$, which is impossible. So it must be that $(r',c'',s)$ is in $P'$ 
for some $c''$. But this means that
$I_2(c'')=I_1(r')^{-1}I_3(s)=I_1(r)^{-1}I_3(s)=I_2(c)$. Thus $c=c''$,
which leads to the same contradiction as before.
\end{proof}

\begin{lemma}\label{l:ruleshiftline}
Let $G$ be a group of order $n$.
Let $P$ be a PLS with rows $R$, columns $C$ and symbols $S$.
Suppose that $P$ contains triples 
$(r,c_1,s_1),\dots,(r,c_\ell,s_\ell)$
with no other triples in row $r$. Let
$P'=P\setminus\{(r,c_i,s_i):1\le i\le\ell\}$ and let $R'$, $C'$ and $S'$
be the rows, columns and symbols of $P'$, respectively.
Let $C_1=\{1\le i\le\ell:c_i\in C'\}$ and
$S_1=\{1\le i\le\ell:s_i\in S'\}$.
If
\begin{enumerate}
\item[\rm (i)] $C_1\cap S_1=\emptyset$,
\item[\rm (ii)] $n\ge |R|+|C_1|(|S'|-1)+|S_1|(|C'|-1)$, and
\item[\rm (iii)] $n\ge |C|+|S|-\ell$,
\end{enumerate}
then $P$ can be embedded in $G$ if $P'$
can be embedded in $G$.
\end{lemma}

\begin{proof}
Consider an embedding $(I_1,I_2,I_3)$ of $P'$ in $G$.
By definition, $P'$ contains at least one triple $(\bullet,c_i,\bullet)$
for each $i\in C_1$.
Likewise, $P'$ contains at 
least one triple $(\bullet,\bullet,s_i)$ for each $i\in S_1$.
Hence the constraint (ii) ensures that 
we can choose a value for $I_1(r)\in G$ outside of the set 
\[
\big\{I_1(r'):(r',c',s')\in P'\big\}
\bigcup\big\{I_3(s)I_2(c_i)^{-1}:i\in C_1,s\in S'\big\}
\bigcup\big\{I_3(s_i)I_2(c)^{-1}:i\in S_1,c\in C'\big\}.
\]
The injectivity of $I_3$ is immediate.
By condition (i) we are free to
define $I_2(c_i)=I_1(r)^{-1}I_3(s_i)$ for $i\in S_1$ and to
define $I_3(s_i)=I_1(r)I_2(c_i)$ for $i\in C_1$.

There are at least $n-(|C'|+|S'|)=n-\big(|C|-(\ell-|C_1|)+|S|-(\ell-|S_1|)\big)$
elements of
\begin{equation}\label{e:badfl}
G\setminus\Big(
\big\{I_2(c'):(r',c',s')\in P'\big\}
\bigcup\big\{I_1(r)^{-1}I_3(s):s\in S'\big\}
\Big)
\end{equation}
So by condition (iii) there are at least $\ell-|C_1|-|S_1|$ elements
of this set. This gives us the option, for 
$i\in\{1,\dots,\ell\}\setminus(C_1\cup S_1)$,  
to choose distinct values for $I_2(c_i)$ in \eref{e:badfl} and put
$I_3(s_i)=I_1(r)I_2(c_i)$. It is routine to check that this yields
an embedding of $P$ in $G$.
\end{proof}

For ease of expression we have stated Lemmas \ref{l:rulenoopt} and
\ref{l:ruleshiftline} in a form that breaks the symmetry between rows,
columns and symbols.  This loses some generality, but we can get it
back by applying the lemmas to each PLS in an orbit under parastrophy. We
will do this without further comment when invoking these lemmas.

The fact that $\cyc(1)=\ab(1)=\kee(1)=1$ is a triviality. The fact
that $\ab(n)=\kee(n)=n-1$ for $n\in\{2,3,4\}$ follows from
\tref{t:evansconj} and the fact that for these orders every species of
quasigroup contains an abelian group. Moreover, for $n\in\{2,3,4\}$
the obstacles for $\ab(n)$ and $\kee(n)$ are precisely those
characterised in \cite{AH83,Dam83}.  Specifically, $\ev_{n,a}$ as
given in \eref{e:evanstight} is an obstacle, and every obstacle
belongs to the species of some $\ev_{n,a}$. However, $\ev_{n,a}$ is
from the same species as $\ev_{n,n-a}$. Thus there are only $\lfloor
n/2\rfloor$ species described by \eref{e:evanstight}.  It is not hard
to check that $\cyc(n)=\ab(n)$ when $n\in\{2,3,4\}$.  The obstacles
for $\cyc(n)$ are the same as for $\ab(n)$ except that there is one
extra obstacle when $n=4$, namely $T_4$ from \eref{e:compmap}.
Henceforth, we may assume that $n\ge 5$.

Let $G$ be any group of order $n\ge5$ and let $P$ be a PLS of size at most
$3$. Lemmas \ref{l:rulenoopt} and \ref{l:ruleshiftline} together show that $P$
cannot be the smallest PLS that does not embed in $G$. It then follows
from \lref{l:lcyc} that $\cyc(n)=\ab(n)=\kee(n)=3$ for all odd $n\ge 5$. 
To confirm the uniqueness of $C_2$ as an obstacle 
we screened the PLS of size $4$ with Lemmas \ref{l:rulenoopt} and 
\ref{l:ruleshiftline}. Most were eliminated immediately (for some,
a quick manual check that they embed in $\Z_5$ was required because
the lemmas only applied for $n\ge7$). 
Apart from $C_2$, the only candidate left standing was:
\begin{equation}\label{e:noninterc}
\left(
\begin{array}{ccc}
a&b&\cdot\\
\cdot&a&b
\end{array}
\right).
\end{equation}
However, this PLS can be embedded in any group
which has an element of order more than 2, which is to say, any group
other than an elementary abelian 2-group, as shown by:
\[
\begin{array}{c|ccc}
&\id&b&b^2\\
\hline
\id&\id&b&\cdot\\
b^{-1}&\cdot&\id&b
\end{array}
\]
Having completed the odd case, from now on we assume that $n$ is even.

Let us next consider the case $n=6$. Two independent computations
confirm that, of the species of PLS of size at most
$6$, only six species do not embed into $\Z_6$. The six species all
have size $6$ and thus these six species are the obstacles for 
$\cyc(6)=\ab(6)=5$. They include the 3 species $\ev_{6,1}$, $\ev_{6,2}$,
$\ev_{6,3}$ from \eref{e:evanstight} and the one species $T_6$ from
\eref{e:compmap}.  Representatives of the other two species are:
\begin{gather}
\left(
\begin{array}{cccccc}
a&\cdot&\cdot&\cdot&\cdot&c \\
\cdot&a&\cdot&\cdot&b&\cdot \\
\cdot&\cdot&b&c&\cdot&\cdot
\end{array}
\right)\label{e:interesting}\\
\left(
\begin{array}{ccc}
a&b&\cdot \\
c&\cdot&b \\
\cdot&c&d
\end{array}
\right)\label{e:nonab}
\end{gather}
These two species deserve individual scrutiny:

\begin{lemma}
The PLS in \eref{e:interesting} does not embed into any group of order~$6$.
\end{lemma}

\begin{proof}
  We assume to the contrary that $(I_1,I_2,I_3)$ embeds
  \eref{e:interesting} into a group $G$ of order $6$, where $I_3$ is the
  identity. Let $\gamma$
  satisfy $I_1(1)I_2(\gamma)=b$. It is clear that $\gamma\in\{2,4\}$.
  The two possible choices for $\gamma$ are equivalent under the
  row-permutation $(2\,3)$, column-permutation $(1\,6)(2\,4)(3\,5)$
  and symbol-permutation $(a\,c)$, so we may assume that $\gamma=2$.

  We may also assume that $I_1(1)=I_2(1)=a=\id$ from which it
  follows that $I_1(2)=b^{-1}$ and
  $b$ must have order~2,~3 or~6. 
  Clearly, $b$ does not have order~2 as this would imply that
  $I_1(2)I_2(1)=b=I_1(2)I_2(5)$, violating the injectivity of $I_2$.

If $b$ has order~3, then $b^2=b^{-1}$ and we have:
\[
\begin{array}{c|cccccc}
& \id&b & x & y & b^2 & c\\
\hline
\id&\id&b&\cdot&\cdot&\cdot&c \\
b^2&b^2&\id&\cdot&\cdot&b&\cdot \\
z&\cdot&\cdot&b&c&\cdot&\cdot
\end{array}
\]
where $G=\{\id,b,b^2,x,y,c\}$ and $z\in G\setminus\{\id,b^2,b,c\}$.
By eliminating other possibilities within the second row we see that
$b^2x=c$ and hence $b^2c=y$. However, this implies that $c=zy=zb^2c$
and thus $z=b$, which is not possible. Therefore $b$ cannot have
order~3.

If $b$ has order~6, then the group is abelian and we may assume that 
$b=1$ giving:
\[
\begin{array}{c|cccccc}
& 0&1 & x & y & 2 & c\\
\hline
0&0&1&\cdot&\cdot&\cdot&c \\
5&\cdot&0&\cdot&\cdot&1&\cdot \\
z&\cdot&\cdot&1&c&\cdot&\cdot
\end{array}
\]
with $\{x,y,c\}=\{3,4,5\}$ and $z\in\{1,2,3,4\}$.
Now, $z+x=1$ and $z+y=c$ so $c+x=y+1$, which has no
solution amongst the available values. 
Therefore $b$ cannot have order~6.
\end{proof}

\begin{lemma} 
The PLS in \eref{e:nonab} does not embed into any abelian group.
\end{lemma}

\begin{proof}
  Assume to the contrary that $(I_1,I_2,I_3)$ embeds this PLS into an
  abelian group. We may assume that $I_1$ maps the rows 
  to $0$, $c$ and $y$, respectively, and $I_2$ maps the columns
  to $0$, $b$ and $x$, respectively.
\[
\begin{array}{c|ccc}
&0&b&x \\
\hline
0&0&b&\cdot \\
c&c&\cdot&b \\
y&\cdot&c&d
\end{array}
\]
We have $y+b=c$ and $c+x=b$ which imply that
$b=y+b+x$. Hence, $d=x+y=0$ which prevents $I_3$ from being injective.
Thus no embedding into an abelian group is possible.
\end{proof}

However, \eref{e:nonab} does embed into each dihedral group 
of order at least~6. Using the presentation
$D_{2k}=\langle r,s | r^k=s^2=\id, sr=r^{-1}s \rangle$, we find this embedding:
\[
\begin{array}{c|ccc}
& \id& rs & s\\
\hline
\id&\id&rs&\cdot \\
r&r&\cdot&rs \\
r^2s&\cdot&r&r^2
\end{array}
\]
In particular, we have established that $\kee(6)=5$ and
there are precisely 5 species
that are obstacles for $\kee(6)$, namely $\ev_{6,1}$, $\ev_{6,2}$,
$\ev_{6,3}$, $T_6$ and \eref{e:interesting}.

With the aid of \lref{l:ruleshiftline} we know that these 5 obstacles for
$\kee(6)$ can be embedded into every group of order at 
least $11$. It follows that $\cyc(n)=\ab(n)=5$ and $\kee(n)=6$ whenever 
$n\equiv0\mod6$ and $n>6$. In this case, the species of the 
PLS in \eref{e:nonab} is the unique obstacle for $\cyc(n)$ and $\ab(n)$.
Characterising the obstacles for $\kee(n)$ requires more
work. Using Lemmas \ref{l:rulenoopt} and \ref{l:ruleshiftline}
and \tref{t:trans} we immediately eliminate all but 50 of the PLS
of size 7. Of these 50 PLS, 42 embed in $\Z_6$ and hence are not obstacles.
Let $\Omega$ be the set of the remaining 8 PLS.
There are 6 PLS in $\Omega$ that contain a PLS from the species
represented by \eref{e:nonab}, 
which explains why they do not embed in $\Z_6$. Two of these 6
are the obstacles that we know from \eref{e:quadcrit}, and the other 4
all embed in $D_6$ and hence are not obstacles. 
That leaves just two PLS in $\Omega$ that we
have not discussed. One of them is
\begin{equation}\label{e:overlapinterc}
\left(
\begin{array}{ccc}
a&b&c\\
b&a&\cdot\\
c&\cdot&a
\end{array}
\right)
\end{equation}
which can embed in any group that has more than one element of order
$2$. In particular, it embeds in $D_{2k}$ for any $k\ge2$ which means
it is not an obstacle for $\kee(2k)$.
The final PLS in $\Omega$ is
\begin{equation}\label{e:order4}
\left(
\begin{array}{ccc}
a&b&c\\
b&c&\cdot\\
c&\cdot&a
\end{array}
\right).
\end{equation}
Suppose $(I_1,I_2,I_3)$ is an embedding of this PLS in a group $G$.
We may assume that $I_1$ maps the rows to $\id,b,c$ respectively, and $I_2$
maps the columns to $\id,b,c$ respectively. It then follows that $b$ is an
element of order $4$. We conclude that $G$ cannot have order $2\mod 4$.
Conversely, it is clear that \eref{e:order4} embeds in $\Z_4$ and hence 
into $\Z_n$ for any $n$ divisible by $4$.
In conclusion, we know that for $n$ divisible by 12 the only obstacles
for $\kee(n)$ are the two species given in \eref{e:quadcrit}. 
When $n=12k+6$ for an integer $k>1$,
there are 3 obstacles as given in \eref{e:quadcrit} and \eref{e:order4}.

It remains to consider orders $n\ge8$ which are divisible by 2 but not by 3.
For such orders, \lref{l:lcyc} shows that $\cyc(n)\le\ab(n)\le\kee(n)\le5$, 
Let $G$ be a group of order $n$. Screening the PLS of size up to $5$
using Lemmas \ref{l:rulenoopt} and \ref{l:ruleshiftline}, we found
only three candidates for the smallest PLS that does not embed in
$G$. The first was $C_2$, which can embed in $G$ by Sylow's Theorem.
The second was $T_5$ from \eref{e:compmap},
which can embed in $G$ by \tref{t:trans}.
The third was the PLS \eref{e:noninterc}, which can be embedded in any 
cyclic group of order greater than $2$, so $\cyc(n)=\ab(n)=\kee(n)=5$.

To find the obstacles for $\kee(n)$, $\ab(n)$ and $\cyc(n)$ 
for $n\equiv2,4\mod6$ we proceed as before.
Using Lemmas \ref{l:rulenoopt} and \ref{l:ruleshiftline}
and \tref{t:trans} we eliminated all but 11 of the PLS
of size 6 (for 11 others we needed to find an embedding in $\Z_8$,
whilst the lemmas took care of all larger groups). The 11 
remaining candidates for obstacles included $C_3$ which we know is an
obstacle for $\kee(n)$, $\ab(n)$ and $\cyc(n)$ and \eref{e:nonab},
which we know is an
obstacle for $\ab(n)$ and $\cyc(n)$ but not for $\kee(n)$.

The remaining 9 PLS can be embedded into cyclic groups of any order at
least~6 and hence are not obstacles for $\kee(n)$, $\ab(n)$ or
$\cyc(n)$. The claimed embeddings for these PLS are shown in the
following, with each PLS embedding as per the bordered table to its
right:
\[
\begin{array}{lll}
\left(
\begin{array}{cccc}
a&b&\cdot&\cdot \\
\cdot&a&\cdot&\cdot \\
\cdot&\cdot&a&b \\
\cdot&\cdot&b&\cdot
\end{array}
\right)
\qquad
\left(
\begin{array}{cccc}
a&b&\cdot&\cdot \\
\cdot&a&c&\cdot \\
\cdot&\cdot&a&b 
\end{array}
\right)
\qquad
\left(
\begin{array}{ccc}
a&b&\cdot \\
\cdot&a&c \\
d&\cdot&a 
\end{array}
\right)
&\phantom{somespace}&
\begin{array}{c|cccc}
& 0&-1 & -3 & -4 \\
\hline
0&0&-1&\cdot&\cdot \\
1&\cdot&0&-2&\cdot \\
3&3&\cdot&0&-1 \\
2 & \cdot&\cdot&-1&\cdot
\end{array}
\\[9ex]
\left(
\begin{array}{cccc}
a&c&\cdot&\cdot \\
\cdot&a&\cdot&b \\
\cdot&\cdot&b&c
\end{array}
\right)
\qquad
\left(
\begin{array}{cccc}
a&\cdot&\cdot&c \\
\cdot&a&\cdot&b \\
c&\cdot&b&\cdot 
\end{array}
\right)
&\phantom{somespace}&
\begin{array}{c|cccc}
& 0&2 & -2 & 1 \\
\hline
0&0&2&\cdot&1 \\
-2&\cdot&0&\cdot&-1 \\
1&1&\cdot&-1&2 
\end{array}
\\[8ex]
\left(
\begin{array}{cccc}
a&b&\cdot&\cdot \\
\cdot&a&b&\cdot \\
\cdot&\cdot&a&b
\end{array}
\right)
\qquad
\left(
\begin{array}{ccc}
a&b&\cdot \\
\cdot&a&b \\
c&\cdot&a 
\end{array}
\right)
&\phantom{somespace}&
\begin{array}{c|cccc}
& 0&-1 & -2 & -3 \\
\hline
0&0&-1&\cdot&\cdot \\
1&\cdot&0&-1&\cdot \\
2&2&\cdot&0&-1 
\end{array}
\end{array}
\]
\[
\begin{array}{lllllll}
\left(
\begin{array}{ccc}
a&b&\cdot \\
\cdot&a&c \\
c&\cdot&b
\end{array}
\right)
&\phantom{space}&
\begin{array}{c|ccc}
& 0&-2 &  -3 \\
\hline
0&0&-2&\cdot \\
2&\cdot&0&1 \\
1&1&\cdot&-2 
\end{array}
&\phantom{somspace}&
\left(
\begin{array}{ccc}
a&\cdot&c \\
\cdot&a&b \\
b&c&\cdot 
\end{array}
\right)
&\phantom{space}&
\begin{array}{c|ccc}
& 0&-1 & 1 \\
\hline
0&0&\cdot&1 \\
1&\cdot&0&2 \\
2&2&1&\cdot 
\end{array}
\end{array}
\]
This completes the proofs of \tref{t:main}, \tref{t:abmain} and
\tref{t:cycmain} and the characterisation of all obstacles.

\section{Concluding remarks}

In the process of answering Problem 3.8 from \cite{DKI} we have solved
three related problems. Namely, we have found, for each order $n$, the
smallest partial latin square that cannot be embedded into (i) any
group of order $n$, (ii) any abelian group of order $n$ and (iii) the
cyclic group of order $n$. We have also identified the unique species
of the smallest PLS that cannot be embedded into an abelian group of
any order, namely \eref{e:nonab}. And we found the two species of PLS
which share the honour of being the smallest that cannot embed into
any group, namely \eref{e:quadcrit}. As a byproduct of our
investigations, we can also be sure that \eref{e:overlapinterc}
represents the unique species of smallest PLS that can embed into some
abelian group but not into any cyclic group. Similar questions had
previously been answered for an important restricted class of PLS
known as a {\em separated, connected latin trades}. Let $\chi$
denote the set of such PLS.  In \cite{CW09} it was found that the
smallest PLS in $\chi$ to not embed in any group has size 11, the
smallest PLS in $\chi$ to embed in some group but not into any
abelian group has size 14, while the smallest PLS in $\chi$ to embed
in some abelian group but not into any cyclic group has size 10.

We close with a discussion of what seems to be an interesting special
case of embedding PLS in groups. By a {\em diagonal PLS} we will mean
a PLS in which cells off the main diagonal are empty. We have already
seen an example in \eref{e:compmap} and, modulo parastrophy,
$\ev_{n,1}$ from \eref{e:evanstight} is another example. The question
of which groups have an embedding for $T_n$ has proved a particularly
deep and fruitful line of enquiry.  It seems that asking the same
question for other diagonal PLS might yield some interesting
results. Equivalently, we may ask the following question for each
given group $G$ of order $n$ and partition $\Pi$ of $n$. Is it
possible to find a permutation $\pi$ of $G$ such that the
multiplicities of the elements of $G$ in the multiset
$\{g\cdot\pi(g):g\in G\}$ form the partition $\Pi$? A theorem of
Hall \cite{Hal52} characterises the possible multisets 
$\{g\cdot\pi(g):g\in G\}$ when $G$ is
abelian. We offer the following extra result as a ``teaser''.

\begin{theorem}
  Let $\Delta$ be the diagonal PLS of size $n$ with $\Delta(i,i)=a$
  for $i\le3$ and $\Delta(i,i)=b$ for $4\le i\le n$. Then $\Delta$ has
  an embedding into a group $G$ of order $n$ if and only if $n$ is
  divisible by $3$.
\end{theorem}

\begin{proof} 
First suppose that $3$ divides $n$ so that $G$ has an element $u$ of order
$3$. By \cite[Lem.~3.1]{DJW15}, there is an embedding $(I_1,I_2,I_3)$ of
$\Delta$ in $G$ with $I_3(a)=u$ and $I_3(b)=\id$.


Next suppose that $\Delta$ has an embedding $(I_1,I_2,I_3)$ in
$G$. By post-multiplying $I_3(a)$, $I_3(b)$ and $I_2(i)$ for $1\le i\le n$
by $I_3(b)^{-1}$, we may assume that
$I_3(b)=\id$. It then follows from \cite[Lem.~3.1]{DJW15} that $G$
contains an element of order $3$.  Hence $G$ has order divisible by
$3$, as required.
\end{proof}

The paper \cite{DJW15} looked at not just whether a diagonal PLS
can be embedded in a given quasigroup, but how many different ways each
such PLS can be embedded. It was shown that there are examples of
quasigroups $Q_1=(Q,\star)$ and $Q_2=(Q,\otimes)$ from different
species that cannot be distinguished by this information. That is,
for each diagonal PLS $D$ and each injection $I_3$ from the symbols of
$D$ to $Q$, there are the same number of embeddings $(I_1,I_2,I_3)$ of
$D$ in $Q_1$ as there are in $Q_2$. A question was posed 
in \cite{DJW15} whether this is possible when $Q_1$ is a group.

\subsection*{Acknowledgements}

We thank Rebecca Stones for independently confirming
the numbers in \eref{e:numPLS}.

 
  \let\oldthebibliography=\thebibliography
  \let\endoldthebibliography=\endthebibliography
  \renewenvironment{thebibliography}[1]{%
    \begin{oldthebibliography}{#1}%
      \setlength{\parskip}{0.2ex}%
      \setlength{\itemsep}{0.2ex}%
  }%
  {%
    \end{oldthebibliography}%
  }

\end{document}